\documentclass[12pt]{amsart}

\usepackage{amssymb, amsfonts, amsthm}
\usepackage{graphicx}

\numberwithin{equation}{section}

\newtheorem{theorem}{Theorem}[section]
\newtheorem{lemma}[theorem]{Lemma}
\newtheorem{corollary}[theorem]{Corollary}

\theoremstyle{definition}

\newtheorem{example}[theorem]{Example}

\newcommand{\A}{{\mathcal A}}

\newcommand{\es}{{\mathcal S}}

\newcommand{\IC}{{\mathbb C}}

\newcommand{\D}{{\mathbb D}}

\newcommand{\CC}{{\mathcal C}}

\newcounter{minutes}
\divide\time by 60
\newcounter{hours}
\multiply\time by 60 \addtocounter{minutes}{-\time}
\begin{document}

\title[Maximal area integral problem]
{Maximal area integral problem for certain class of univalent analytic functions}

\author{Saminathan Ponnusamy 
}
\address{S. Ponnusamy,
Indian Statistical Institute (ISI), Chennai Centre, SETS (Society
for Electronic Transactions and security), MGR Knowledge City, CIT
Campus, Taramani, Chennai 600 113, India. }
\email{samy@isichennai.res.in, samy@iitm.ac.in}

 \author{Swadesh Kumar Sahoo 
 }
 \address{S. K. Sahoo, Discipline of Mathematics,
 Indian Institute of Technology Indore,
 Indore 452 017, India}
 \email{swadesh@iiti.ac.in}

 \author{Navneet Lal Sharma}
 \address{N. L. Sharma, Discipline of Mathematics,
 Indian Institute of Technology Indore,
 Indore 452 017, India}
 \email{sharma.navneet23@gmail.com}


\keywords{Analytic, univalent, convex,   starlike functions and
spirallike functions, Dirichlet-finite, area integral, and
Gaussian hypergeometric functions. \\
$
^\dagger$ {\tt  The first author is on leave from the Department of Mathematics,
Indian Institute of Technology Madras, Chennai-600 036, India}
}

\begin{abstract}
One of the classical problems concerns the class of analytic functions $f$ on the open unit disk $|z|<1$ which have
finite Dirichlet integral $\Delta(1,f)$, where
$$\Delta(r,f)=\iint_{|z|<r}|f'(z)|^2 \, dxdy \quad (0<r\leq 1).
$$
The class $\es ^*(A,B)$ of normalized functions $f$ analytic in $|z|<1$ and satisfies the
subordination condition $zf'(z)/f(z)\prec (1+Az)/(1+Bz)$ in $|z|<1$ and for some $-1\leq B\leq 0$,  $A\in  \IC$ with $A\neq B$,
has been studied extensively. In this paper, we solve the extremal problem of determining  the value of
$$\max_{f\in \es ^*(A,B)}\Delta(r,z/f)
$$
as a function of $r$. This settles the question raised by Ponnusamy and Wirths in \cite{PW13}.
One of the particular cases includes solution to a conjecture of Yamashita which was settled recently
by Obradovi\'{c}  et. al \cite{OPW14}.

\end{abstract}

\maketitle \pagestyle{myheadings} \markboth{S. Ponnusamy, S. K. Sahoo, and  N. L. Sharma}{Maximal area integral problem}


\begin{center}
\texttt{File:~\jobname .tex, printed: \number\day-\number\month-\number\year,
\thehours.\ifnum\theminutes<10{0}\fi\theminutes}
\end{center}

\thispagestyle{empty}


\section{Introduction and Preliminaries} \label{sec1}
Let $f$ and $g$ be two analytic functions in the unit disk $\D:=\{z\in \mathbb{C}: |z|<1 \}$.
We say that $f$ is {\em subordinate} to $g$, written as $f \prec g$,
if there exists an analytic function $w: \D \rightarrow \overline{\D}$ with
$w(0)=0$ such that $f(z)=g(w(z))$ for $ z\in \D.$
In particular, if $g$ is univalent in $\D$, then $f \prec g$ is equivalent to
$f(\D)\subset g(\D)$ and $f(0)=g(0)$; see \cite{Dur83,Good83}.

Denote by $\A$ the class of functions of the form
$f(z)=z+\sum_{n=2}^\infty a_n z^n
$
analytic in  $\D $. The class of univalent functions in $\A$ is denoted by $\es $.
Two subclasses of $\es $ to which we will make frequent reference are ${\es}^*$
and $\CC$, the subclasses of starlike functions (with respect to the origin) and convex
functions, respectively. Recall that a function
$f\in {\es}^*$ is characterized by the condition
$$\frac{zf'(z)}{f(z)} \prec \frac{1+z}{1-z},\quad z\in \D.
$$
There are a number of ways the class ${\es}^*$ has been generalized in the literature and
one such generalization is defined as follows: For $-1\leq B\leq 0$ and $A\in  \IC$, $A\neq B$,
define
$$\es ^*(A,B):=\left \{f\in \A: \frac{zf'(z)}{f(z)}\prec \frac{1+Az}{1+Bz},
\quad   z\in \D   \right\}.
$$
The function $k_{A,B}$ defined by
\begin{equation}\label{eq2-ext}
\displaystyle k_{A,B}(z):=
\left \{
\begin{array}{ll}
ze^{Az} & \mbox{ if } B=0,\\
z(1+Bz)^{(A/B)-1} & \mbox{ if } B\neq 0
     \end{array}\right.
\end{equation}
is in $\es ^*(A,B)$ and acts the role of extremal function for the class $\es ^*(A,B)$.
However, if $A=e^{i\alpha}(e^{i\alpha}-2\beta \cos \alpha)$ with $\beta<1$ and $B=-1$, then $\es ^*(A,B)$ reduces to the class
$\mathcal{S}_\alpha (\beta ) $ of functions $f$ (called $\alpha$-spirallike of order $\beta$) satisfying the condition
$$ {\rm Re}  \left ( e^{-i\alpha}\frac{zf'(z)}{f(z)}\right )>\beta \cos \alpha, \quad z\in \D,
$$
and recall that each function in $\mathcal{S}_\alpha (\beta )$ is univalent if $ \beta \in [0,1)$ and $\alpha \in (-\pi/2, \pi/2)$
(see \cite{Lib67}). Clearly, $\mathcal{S}_\alpha (\beta )\subset \mathcal{S}_\alpha (0)$ whenever $0\leq \beta <1$.
Functions in $\mathcal{S}_\alpha(0)$  are called
$\alpha$-spirallike. The class $\mathcal{S}_\alpha (0)$ was introduced by  ${\rm \check{S}}$pa${\rm\check{c}}$ek \cite{Spacek-33}
and the set $\mathcal{S}p=\cup \{\mathcal{S}_\alpha (0):\,\alpha \in (-\pi/2, \pi/2) \}$ is referred to us the
class of spirallike functions.  As remarked in
\cite{Lib67}, spirallike functions have been used to obtain important counter-examples in geometric function theory
(see also \cite[p.~72 and Theorem 8.11]{Dur83}).

More often, the class $\es ^*(A,B)$ is studied with the restriction $-1\leq B<A\leq 1$ (see Janowski \cite{Jan73})
so that the values of ${zf'(z)}/{f(z)}$ lie inside the disk in the right half plane with center $(1-ABr^2)/(1-B^2r^2)$ and
radius $(A-B)r/(1-B^2r^2)$,  and so,  the class $\es ^*(A,B)$ becomes a subclass of ${\es}^*$ whenever $-1\leq B<A\leq 1$.
We here list down in Table~1 the certain basic subclasses of the class ${\es}^*$ that are studied
for various choices of the pair $(A,B)$. Set for an abbreviation $p(z):=zf'(z)/f(z).$
\begin{center}
{\small\begin{tabular}{|c|l|l|l|l|l|}
\hline
 \textbf{Year} & \textbf{Authors} &  \textbf{ $\es ^*(A,B)$}
             & \textbf{Conditions} & \textbf{Subordination form} \\
\hline
1921 & Nevanlinna \cite{Nev21} &  $\es ^*(1,-1)={\es}^*$ & ${\rm Re}\, p(z)>0$
     &  $\displaystyle p(z)\prec\frac{1+z}{1-z}$\\
\hline
1936 & Robertson \cite{Rob36}  & $\es ^*(1-2\beta,-1)$ & ${\rm Re}\, p(z)>\beta$
     &  $\displaystyle p(z)\prec\frac{1+(1-2\beta) z}{1-z}$\\
     &  &  $={\es}^*(\beta),\,\beta \in [0,1)$      &     & \\
\hline
1968 &  Singh \cite{Sin68}    & $\es ^*(1,0)$ & $\displaystyle |p(z)-1|<1$
     & $\displaystyle p(z)\prec 1+z$\\
\hline
1968 & Padamanabhan \cite{Pad68}  & $\es ^*(\alpha,-\alpha) :={\mathcal T} (\alpha),$   & $\displaystyle
\left|\frac{p(z)-1}{p(z)+1}\right|<\alpha$
& $\displaystyle p(z)\prec\frac{1+\alpha z}{1-\alpha z}$\\
&             & $\alpha \in (0,1]$        &     & \\
\hline
1974 & Singh and    & $\es ^*(1,\frac{1}{\alpha}-1),\,\alpha \geq
\frac{1}{2}$ & $\displaystyle |p(z)-\alpha|<\alpha$
& $\displaystyle p(z)\prec \frac{1+z}{1+\frac{1-\alpha}{\alpha}z}$\\
  &   Singh \cite{SiSi74}         &        &     & \\
\hline
1978 & Silverman \cite{Sil78}   & $ \es ^*\left(\frac{b^2-a^2+a}{b},\frac{1-a}{b}\right)$,
      & $\displaystyle |p(z)-a|<b$  & $\displaystyle p(z)\prec
      \frac{1+\frac{b^2-a^2+a}{b}z}{1+\frac{1-a}{b}z}$\\
      &           & $a+b\geq 1~\&$          &     & \\
       &           & $a\in [b,1+b]$         &     & \\
\hline
2014 &  Sahoo  and & $\es ^*((1-2\beta)\alpha,-\alpha)$
& $\displaystyle  \left|\frac{p(z)-1}{p(z)+1-2\beta}\right|$
& $\displaystyle p(z)\prec\frac{1+(1-2\beta)\alpha z}{1-\alpha z}$\\
 & Sharma \cite{SS14}   & $:={\mathcal T} (\alpha,\beta),$ & $< \alpha$ & \\
 &       & $\alpha \in (0,1],\,\beta \in [0,1)$      &  & \\
\hline
\end{tabular}}

\vspace*{0.2cm} Table~1
\end{center}
From (\ref{eq2-ext}), we note that
\begin{equation}\label{eq3}
k_{1-2\beta,-1}(z)= \frac{z}{(1-z)^{2(1-\beta)}}:= k_{\beta}(z),\, \mbox{ and }\, k_{\alpha,-\alpha}(z)
=\frac{z}{(1-\alpha z)^{2}}.
\end{equation}

Suppose that $f$ is a function analytic in $\D$. We denote
by $\Delta(r,f)$, the area of the multi-sheeted image of
the disk $\mathbb{D}_r:=\{z\in\mathbb{C}:\,|z|<r\}$ ($0<r\le 1$) under $f$. Thus,
in terms of the coefficients of $f$, $f'(z)=\sum_{n=1}^\infty n a_nz^{n-1}$, one gets
with the help of the classical Parseval-Gutzmer formula (see \cite{SS14}) the relation
\begin{equation}\label{ex-eq1}
\Delta(r,f)=\iint_{\mathbb{D}_r}|f'(z)|^2 \, dxdy = \pi \sum_{n=1}^\infty n |a_n|^2r^{2n}
\end{equation}
which is called the Dirichlet integral of $f$.
Computing this area is known as the {\em area problem for the functions of type $f$.}
Thus a function has a finite Dirichlet integral exactly when its image has finite area
(counting multiplicities).  All polynomials and,
more generally, all functions $f\in\A$ for which $f'$ is bounded on $\D$ are Dirichlet finite.

Our work in this paper is motivated by the work of Yamashita~\cite{Yam90} and
recent works ~\cite{OPW13,OPW14,PW13,SS14}. In 1990, Yamashita~\cite{Yam90} conjectured that
\begin{equation}\label{eq3-ext}
\displaystyle \max_{f\in \CC}\Delta \left(r,\frac{z}{f}\right)=\pi r^2
\end{equation}
for each $r,\, 0<r\leq 1$. The maximum is attained only by the rotations of
the function $l(z)=z/(1-z).$
In 2013, the Yamashita conjecture was settled in \cite{OPW13} (see Corollary~\ref{cor1}) in a
more general setting including functions for functions from $\mathcal{S}_\alpha (\beta )$ (see \cite{PW13}).
In the recent paper \cite{SS14}, the maximum area problem for
the functions of type $z/f(z)$ when $f\in {\mathcal S}^* ((1-2\beta)\alpha,-\alpha)\equiv
{\mathcal T} (\alpha,\beta)$ $(\,0< \alpha \leq 1,~ 0\leq \beta <1\,)$, is established
(see Corollary~\ref{cor6}).

A general problem on the Yamashita conjecture for the class $\es ^*(A,B)$ was suggested in \cite{PW13}
(see also \cite{OPW13,OPW14}), and partially it is solved in \cite{SS14}. In this paper, we solve the problem in complete generality
for the full class $\es ^*(A,B)$, and the main results are stated in Section \ref{sec-main}.

\section{Main Theorems}\label{sec-main}
For complex numbers $a,b$ and $c$ with $c$ neither zero nor a negative integer,
the Gaussian hypergeometric series ${}_2F_1(a,b;c;z)$ is defined by
$${}_2F_1(a,b;c;z)=\sum_{n=0}^{\infty}\frac{(a)_n(b)_n}{(c)_n}\frac{z^n}{n!}, \quad |z|<1.
$$
Clearly, the function ${}_2F_1(a,b;c;z)$ is analytic in $\D$ and  thus, the shifted function $z{}_2F_1(a,b;c;z)$ belongs to $\A$.
Here $(a)_n$ denotes the shifted factorial notation defined, in terms of the
Gamma function, by
$$(a)_n=\frac{\Gamma (a+n)}{\Gamma (a)}=
\left \{
\begin{array}{ll}
\displaystyle{a(a + 1) \cdots (a + n-1)} & \mbox{ if } n\geq 1; \\[2mm]
1 & \mbox{ if } n=0,\, a\neq 0.
\end{array}
\right.
$$
The asymptotic behaviour of $F(a,b;c;z)$ near $z=1$ reveals that
$$ F(a,b;c;1) = \frac{\Gamma (c)\Gamma
(c-a-b)}{\Gamma (c-a)\Gamma (c-b)}< \infty \quad \mbox{ for ${\rm Re}\, c>{\rm Re}\,(a+b)$}.
$$
If either (or both) of $a$ and $b$ is (are) zero or a negative integer(s), then the power series reduces to a polynomial;
see~\cite{Rai60}.  Similarly, ${}_0F_1(a;z)$ is defined as
$${}_0F_1(a;z)= \sum_{n=0}^{\infty}\frac{1}{(a)_n}\frac{z^n}{n!}, \quad |z|<1.
$$

We now state our main results and their proofs will be given in Section \ref{sec3}.

\begin{theorem}\label{thm1}
Let $f(z)\in \es ^*(A,0)$ for  $0< |A|\leq 1$. 
Then  for $r, \, 0<r\leq 1$, we have
$$ \max_{f\in \es ^*(A,0)} \Delta \left(r,\frac{z}{f}\right)  =E_{A,0}(r)
$$
where $E_{A,0}(r)=\pi |A|^2{r}^2 {}_0F_1(2;|A|^2{r}^2)$.
The maximum is attained by the rotations of the function $ k_{A,0}(z)=ze^{Az}$.
\end{theorem}

If ${A=1}$ in Theorem~\ref{thm1}, then we get
\begin{corollary}\label{cor5}
Let $f\in \es ^*(1,0)$. Then we have
$$\max_{f\in \es ^*(1,0)}~\Delta \left(r,\frac{z}{f}\right)
=\pi r^2 {}_0F_1(2;r^2) \quad \mbox{for $0<r\leq 1$},
$$
where the maximum is attained only by the rotation of the function $k_{1,0}(z)=ze^{z}$.
\end{corollary}


\begin{theorem}\label{thm2}
Let $f\in \es ^*(A,B)$ for  $-1\leq B<0$ and $A \neq B  $ and $z/f(z)$ be a non-vanishing analytic function in $\D$.
Then, for $0<r\leq 1$, we have
\begin{align*}
\displaystyle \max_{f\in \es ^*(A,B)} \Delta \left(r,\frac{z}{f}\right)
& = E_{A,B}(r) =\pi |\overline{A}-B|^2{r}^2 {}_2F_1(A/B,\overline{A}/B;2;{B}^2{r}^2).
\end{align*}
 The maximum is attained for the rotations
of the function $ k_{A,B}(z)$ defined by $(\ref{eq2-ext}).$
\end{theorem}

Note that Theorems~\ref{thm1} and \ref{thm2} generalize the results proved in \cite{OPW13,SS14,Yam90}.
To see the bounds for the Dirichlet finite function, we write
$$E_{A,0}(1)= \pi |A|^2\sum_{n=0}^{\infty}\frac{1}{(1)_{n}(2)_n}{|A|}^{2n}
~\mbox{ and }~
E_{A,B}(1)= \pi |\overline{A}-B|^2\sum_{n=0}^{\infty}\frac{(A/B)_{n} (\overline{A}/B)_{n} }{(1)_{n}(2)_n}{B}^{2n}.
$$
For certain values of $A$ and $B$, the images of the unit disk under the extremal functions
$g_{A,0}(z)=z/k_{A,0}(z)=e^{-Az}$ and $g_{A,B}(z)=z/k_{A,B}(z)=(1+Bz)^{1-A/B}$,
and numerical values of $E_{A,0}(1)$ and $E_{A,B}(1)$
are described in Figures~1--4 and Table~2, respectively. We remind the reader that for $B=-1$,
$E_{A,B}(1)$ is finite only if $2>{\rm Re}((A+\overline{A})/B)$, i.e. if ${\rm Re}\,A>-1$.

\begin{figure}
\begin{minipage}[b]{0.5\textwidth}
\includegraphics[width=6.5cm]{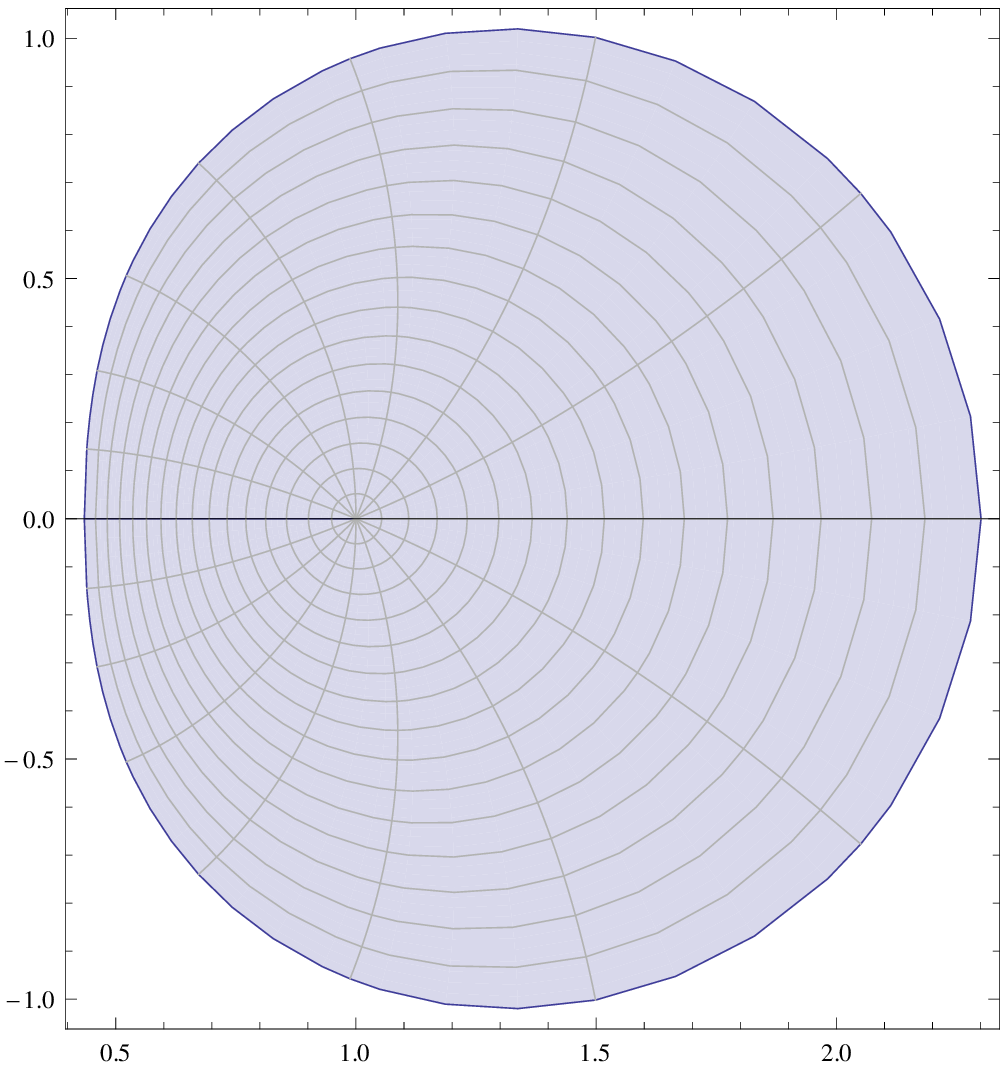}
\hspace*{1.3cm}The image domain $g_{5/6,0}(\D)$
\end{minipage}
\begin{minipage}[b]{0.45\textwidth}
\includegraphics[width=7cm]{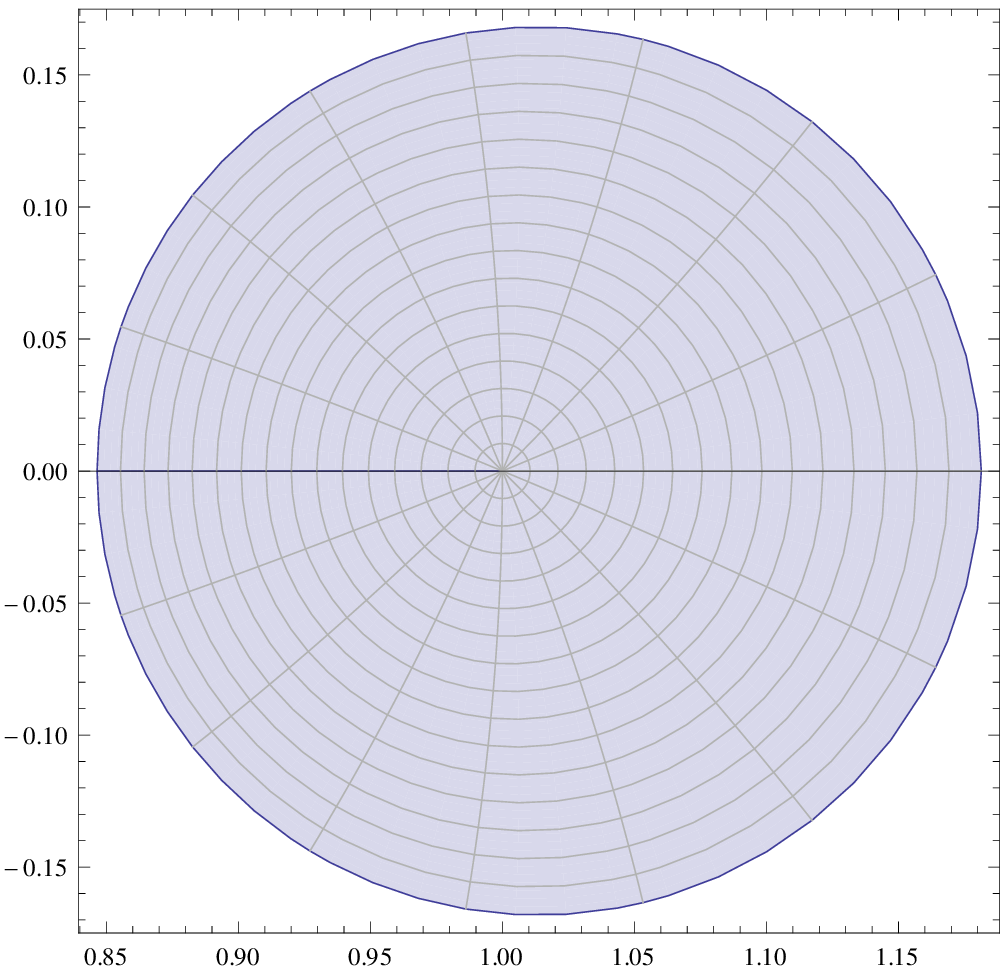}
\hspace*{1.3cm} The image domain $g_{1/6,0}(\D)$
\end{minipage}
\caption{Images of the unit disk under $g_{5/6,0}$ and $g_{1/6,0}$.}
\end{figure}

\begin{figure}
\begin{minipage}[b]{0.5\textwidth}
\includegraphics[width=6.6cm]{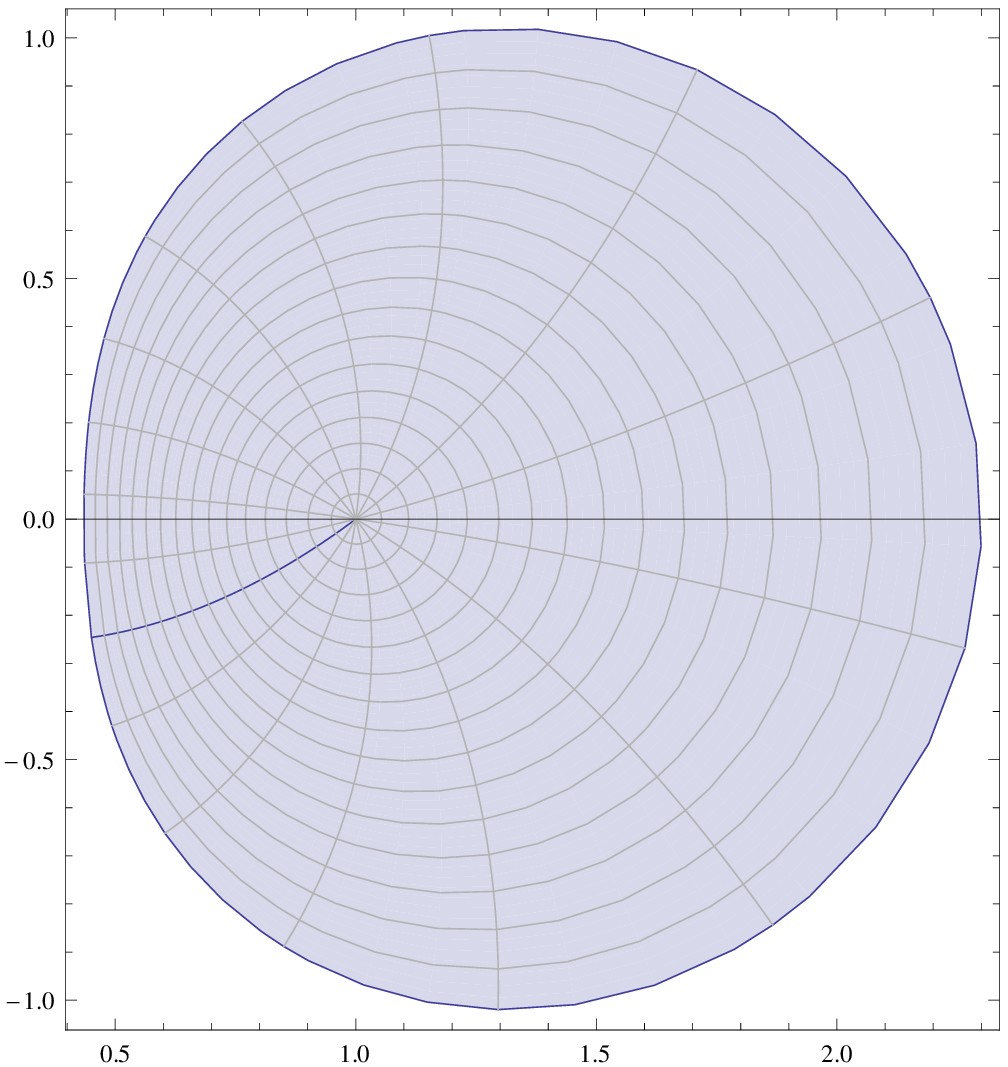}
\hspace*{1cm}The image domain $g_{2/3+i/2,0}(\D)$
\end{minipage}
\begin{minipage}[b]{0.45\textwidth}
\includegraphics[width=6.7cm]{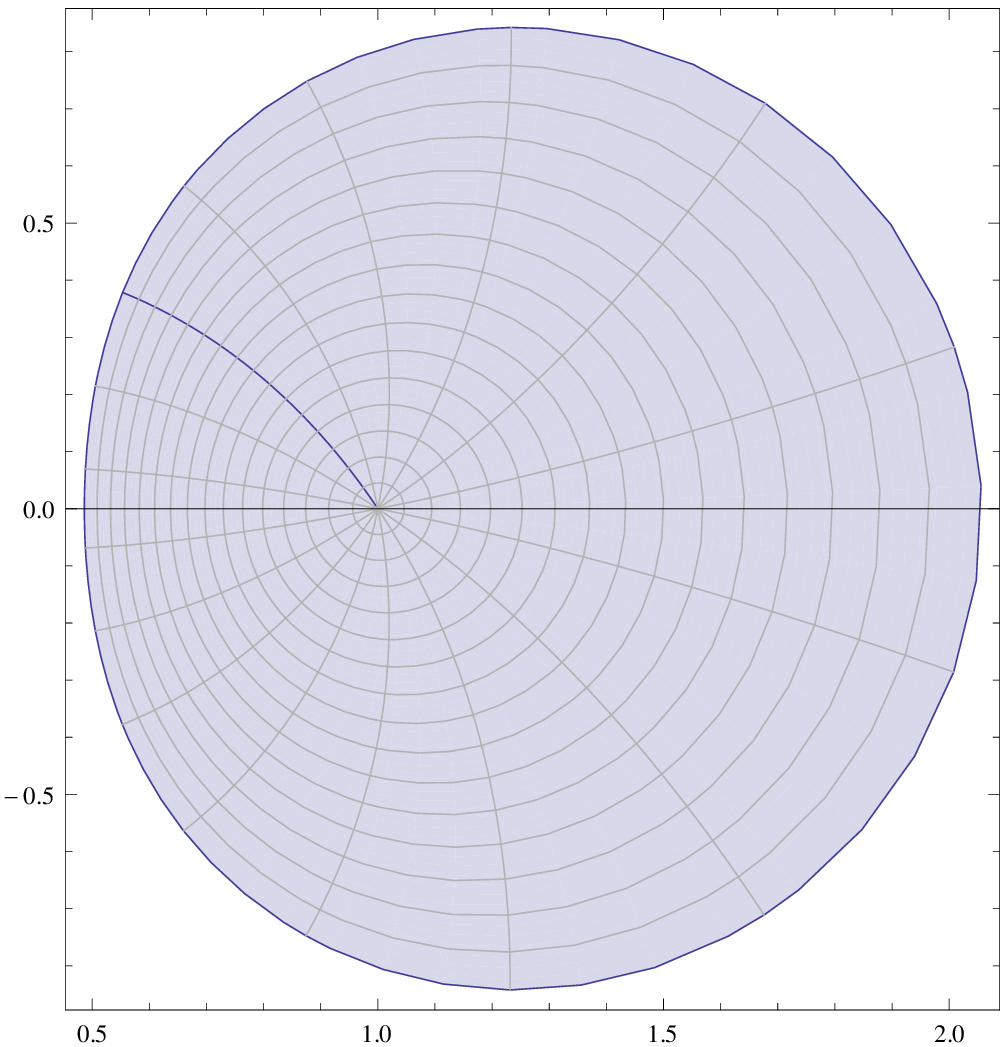}
\hspace*{0.7cm} The image domain $g_{(2-3i)/5,0}(\D)$
\end{minipage}
\caption{Images of the unit disk under $g_{2/3+i/2,0}$ and $g_{(2-3i)/5,0}$.}
\end{figure}

\begin{figure}
\begin{minipage}[b]{0.5\textwidth}
\includegraphics[width=6.8cm]{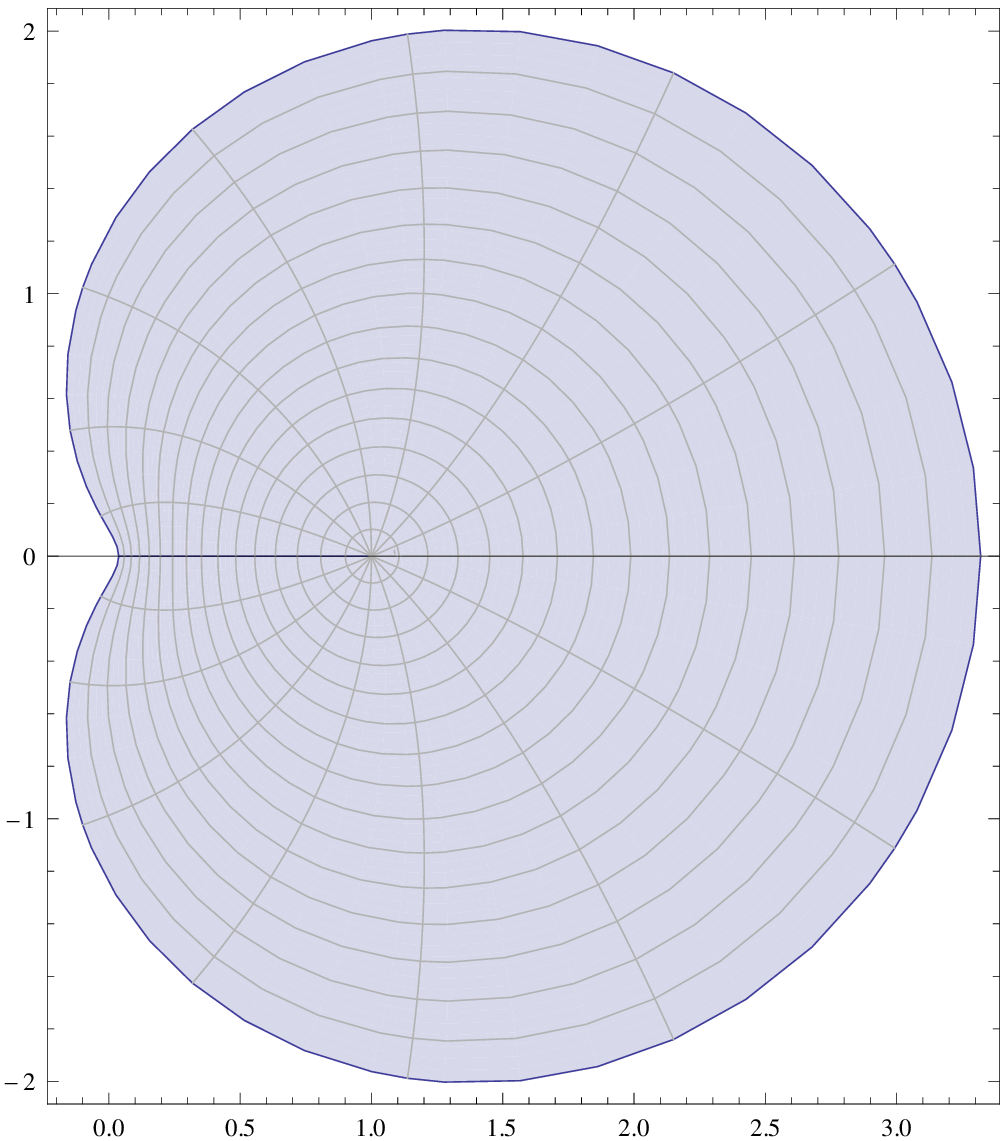}
\hspace*{1cm}The image domain $g_{5/6,-4/5}(\D)$
\end{minipage}
\begin{minipage}[b]{0.45\textwidth}
\includegraphics[width=7.5cm]{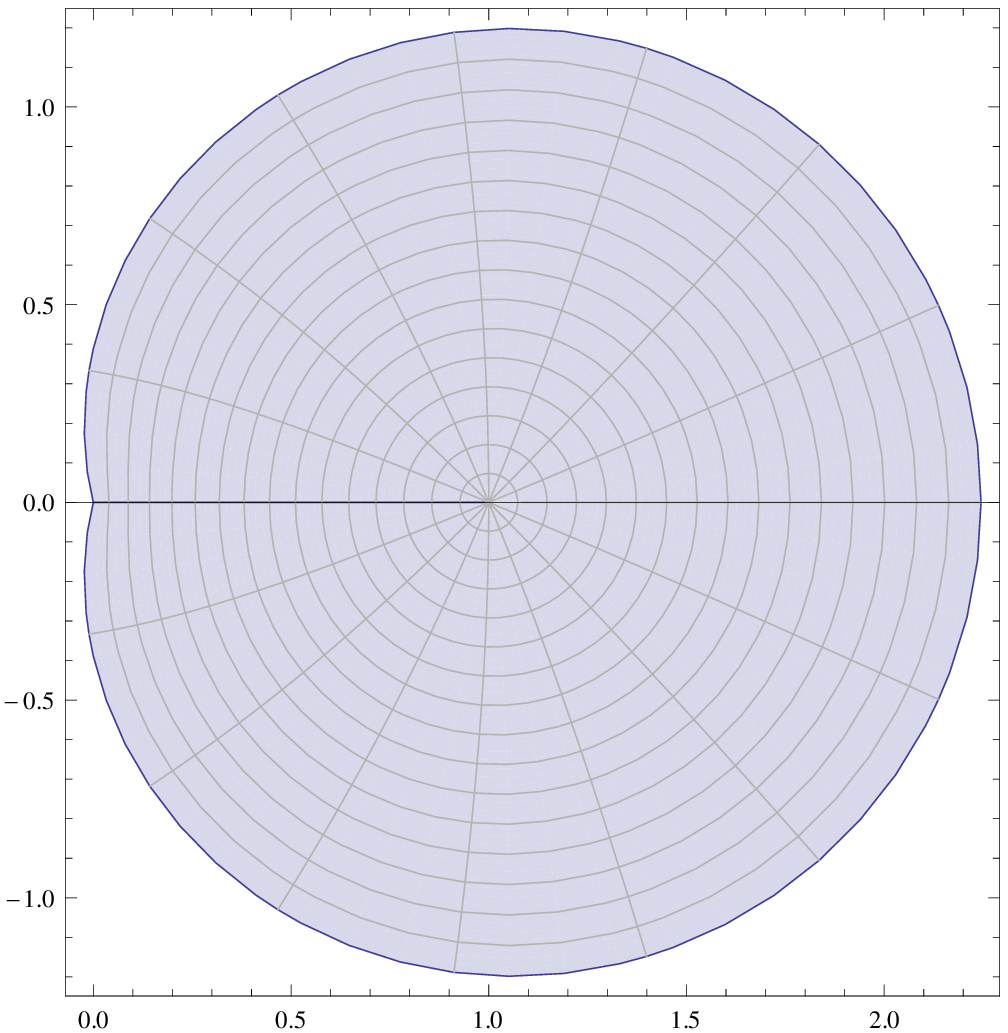}
\hspace*{1cm} The image domain $g_{1/6,-1}(\D)$
\end{minipage}
\caption{Images of the unit disk under $g_{5/6,-4/5}$ and $g_{1/6,-1}$.}
\end{figure}

\begin{figure}
\begin{minipage}[b]{0.5\textwidth}
\includegraphics[width=7.1cm]{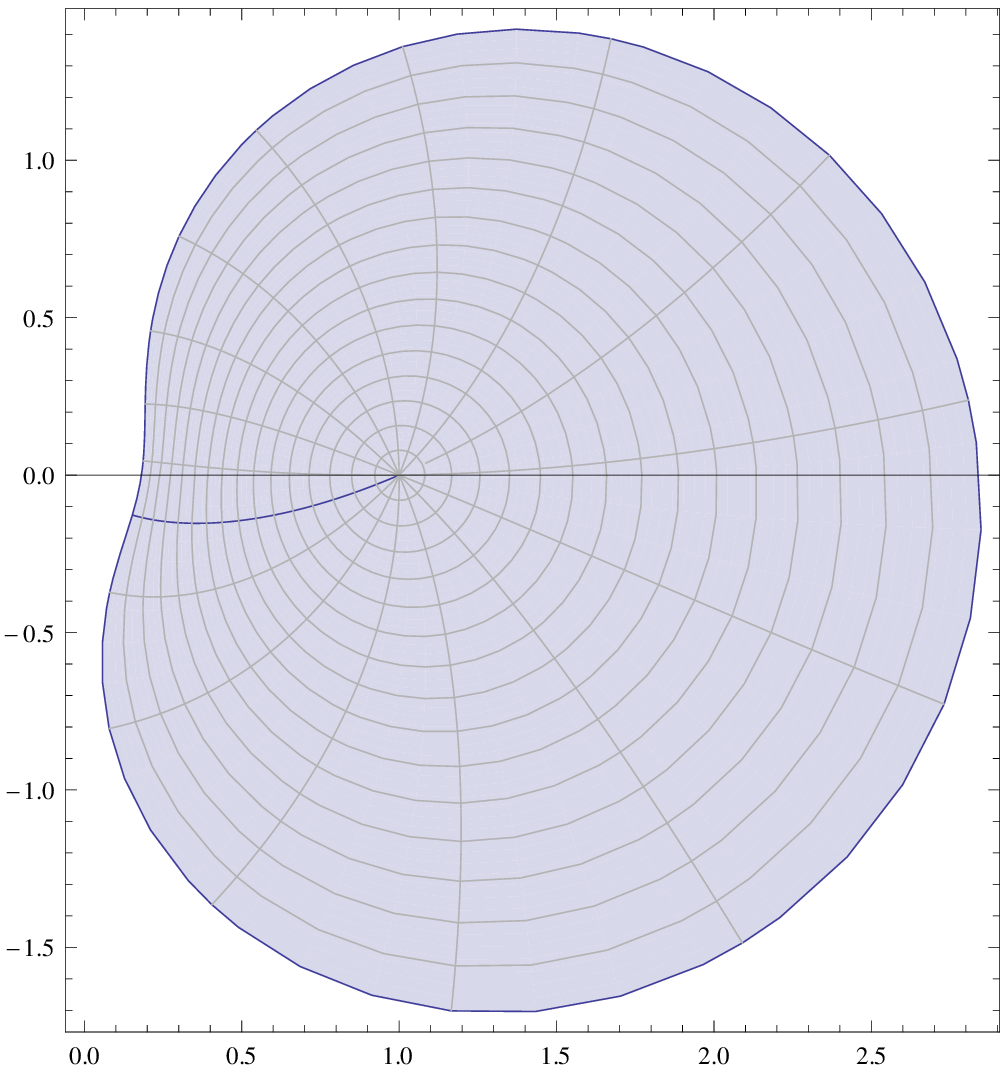}
\hspace*{1cm}The image domain $g_{2/3+i/2,-1/2}(\D)$
\end{minipage}
\begin{minipage}[b]{0.44\textwidth}
\includegraphics[width=7.2cm]{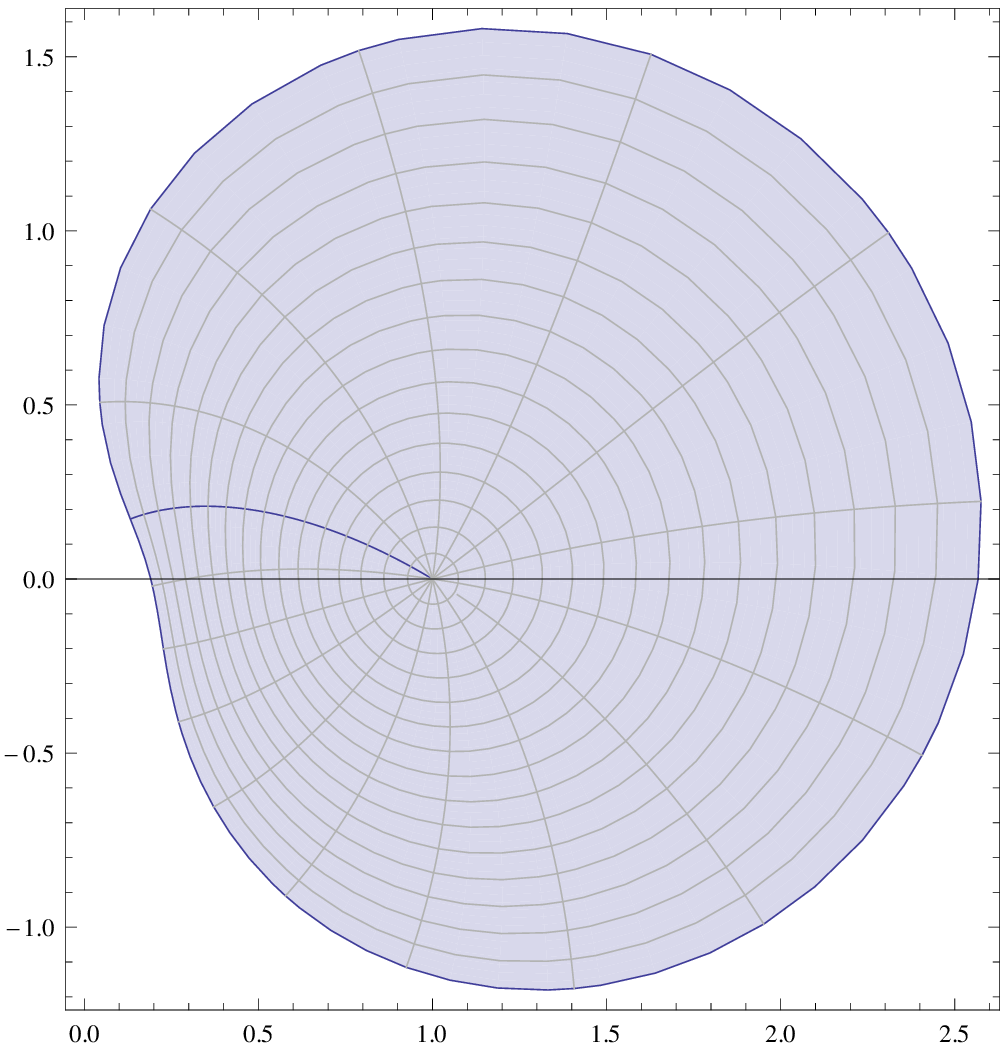}
\hspace*{0.7cm} The image domain $g_{(2-3i)/5,-3/5}(\D)$
\end{minipage}
\caption{Images of the unit disk under $g_{2/3+i/2,-1/2}$ and $g_{(2-3i)/5,-3/5}$.}
\end{figure}

\begin{center}
\begin{tabular}{|c|c|c|c|}
\hline
\textbf{$A$} & \textbf{Approximate Values of } & $B$ & \textbf{ Approximate Values of} \\
  & \textbf{ $E_{A,0}(1)$} & & \textbf{ $E_{A,B}(1)$} \\
\hline
 $5/6$ & $3.03211$ & $-4/5$ & $11.2917$ \\
\hline
 $1/6$ & $0.0884841$ & $-1$ & $4.34607$\\
\hline
 $2/3+i/2$ & $3.03211$ & $-1/2$ & $6.90284$\\
\hline
$(2-3i)/5$ & $2.09682$ & $-3/5$ & $5.4645$ \\
\hline
\end{tabular}

\vspace*{0.2cm}
Table~2
\end{center}

We now state certain consequences of Theorem~\ref{thm2} for the several of the special classes introduced by several authors (refer Table~1).
It is a simple exercise to see that M\"{o}bius transformation $w=\phi (z)$ defined by
$$ w=\phi (z)=\frac{1+Az}{1+Bz}
$$
maps the unit disk $\D$ onto the half-plane
$$ {\rm Re\,} ((1+\overline{A})w) >\frac{1-|A|^2}{2}
$$
whenever $B=-1$ and $A\neq 1$. In particular, if $A=e^{i\alpha}(e^{i\alpha}-2\beta \cos\alpha)$ $(\beta <1)$, then as
remarked in the introduction, the last condition reduces to
$$ {\rm Re\,} (e^{-i\alpha}w)>\beta \cos \alpha.
$$
If $-1<B\leq 0$ and $A\neq B$, then $\phi$ maps  $\D$ onto the disk
$$\left |w -\frac{1-\overline{A}B}{1-B^2} \right | <\frac{|A-B|}{1-B^2}.
$$
This observation helps us to formulate important  special cases.

If we choose $A=1-2\beta$ and $B=-1$ in Theorem~\ref{thm2}, we get

\begin{corollary}\label{cor1}\cite[Theorem 3]{OPW13}
Let $f\in {\es}^*(\beta)$ for some $0\leq \beta <1$.
Then we have
$$\displaystyle \max_{f\in {\es}^*(\beta)}\Delta \left(r,\frac{z}{f}\right)
=4\pi (1-\beta)^2r^2 {}_2F_1(2\beta-1,2\beta-1;2;r^2) \quad \mbox{for $0<r\leq 1$},
$$
where the maximum is attained  by the rotation of the function
$k_{\beta}(z)=\frac{z}{(1-z)^{2(1-\beta)}}$.
\end{corollary}

We remark that when $A=1$ and $B=-1$, Theorem~\ref{thm2} turns  into
\cite[Theorem~A]{OPW13}.
If ${A=(1-2\beta)\alpha}$ and $B=-\alpha$ in Theorem~\ref{thm2}, then we get

\begin{corollary}\label{cor2}\cite[Theorem~1.3]{SS14}
Let $f\in {\mathcal S}^* ((1-2\beta)\alpha,-\alpha)={\mathcal T} (\alpha,\beta)$  for  $0<\alpha\leq 1$ and $0\leq \beta <1$.
Then we have
$$\max_{f\in {\mathcal T}  (\alpha,\beta) } \Delta \left(r,\frac{z}{f}\right)
= 4\pi{\alpha}^2(1-\beta)^2{r}^2 {}_2F_1(2\beta-1,2\beta-1;2;{\alpha}^2
{r}^2),\quad |z|<r
$$
for all $r, \, 0<r\leq 1$. The maximum is attained  by the rotation
of the function $ k_{(1-2\beta)\alpha,-\alpha}(z)$ defined by \eqref{eq2-ext}.
\end{corollary}

The case $\beta =0$ of Corollary \ref{cor2} (i.e. ${A=\alpha}$ and $B=-\alpha$ of Theorem~\ref{thm2}) gives

\begin{example}\label{cor3}\cite[Theorem~3.1]{SS14}
If $f\in {\mathcal S}^* (\alpha,-\alpha):={\mathcal T} (\alpha)$ for some $0<\alpha \leq 1$, then
one has
$$\max_{f\in {\mathcal T} (\alpha)}\Delta \left(r,\frac{z}{f}\right)=2\pi {\alpha}^2r^2(2+{\alpha}^2r^2)
$$
for all $r,\, 0<r\leq 1$, and the maximum is attained by the rotation of the function
$k_{\alpha,-\alpha}(z)=z/(1-\alpha z)^2$.
\end{example}

If we choose $A=e^{i\alpha}(e^{i\alpha}-2\beta \cos\alpha)$ and $B=-1$ in Theorem~\ref{thm2}, we get \cite[Theorem 3]{PW13}.


If $A=1$ and $B=(1-\alpha)/\alpha,\, \alpha \geq 1/2$, then Theorem~\ref{thm2} yields
\begin{corollary}\label{cor6}
If $\alpha \geq 1/2,\, f\in \es ^*(1,(1-\alpha)/\alpha)$,  
then  we have
$$\max_{f\in \es ^*(1,(1-\alpha)/\alpha)}~\Delta \left(r,\frac{z}{f}\right)
=\pi  \left(2-\frac{1}{\alpha}\right)^2 r^2
{}_2F_1\left(\frac{\alpha}{1-\alpha},\frac{\alpha}{1-\alpha};2;\left(\frac{\alpha-1}{\alpha}\right)^2r^2\right),
$$
for $0<r\leq 1$, where the maximum is attained   by the rotation of the function
$k_{1,(1-\alpha)/\alpha}(z)$ defined by $(\ref{eq2-ext}).$
\end{corollary}
If ${A=(b^2-a^2+a)/b}$ and $B=(1-a)/b$ with $a+b\geq 1,\,a\in [b,1+b]$, as a consequence of
Theorem~\ref{thm2} we obtain the following maxima area problem for functions in a class introduced by Silverman
(see Table~1 for the reference).
\begin{corollary}\label{cor7}
Let $ f\in \es ^*((b^2-a^2+a)/b,((1-a)/b))$.
Then we have
$$\max_{f\in \es ^*((b^2-a^2+a)/b,((1-a)/b))}~\Delta \left(r,\frac{z}{f}\right)
=\pi  s_1^2 r^2{}_2F_1\left(s_2,s_2;2;\left(\frac{1-a}{b}\right)^2r^2\right)
\quad \mbox{for $0<r\leq 1$},
$$
where $s_1=(b^2-a^2+2a-1)/b$ and $ s_2=(b^2-a^2+a)/(1-a);\, a+b\geq 1,\,a\in [b,1+b].$
The maximum is attained  by the rotation of the function $k_{(b^2-a^2+a)/b,(1-a)/b)}(z)$
defined by $(\ref{eq2-ext})$.
\end{corollary}


In Section~\ref{sec2}, we present useful lemmas which are the main tools to prove our main theorems.


\section{Preparatory Results}\label{sec2}
If $f\in {\mathcal A}$ such that $z/f(z)$ is non-vanishing in $\D$ (eg. the non-vanishing condition is
ensured whenever $f\in \es$), then
\begin{equation}\label{eq-1-sec.2}
\frac{z}{f(z)}=1+\sum_{n=1}^{\infty}b_nz^n, \, z\in \D.
\end{equation}
We first present a necessary coefficient condition for a function $f$ of the form (\ref{eq-1-sec.2})
to be in $\es ^*(A,B).$
\begin{lemma}\label{lem1}
Let $f\in \es ^*(A,B)$ for  $-1\leq B\leq 0$ and $A \neq B$ and $f$ be of the form $(\ref{eq-1-sec.2})$.
Then
$$\sum_{k=1}^{\infty} \left(k^2-|B-A-kB|^2\right)|b_k|^2 \leq |A-B|^2
$$
holds.
\end{lemma}
\begin{proof}
Denote by $g(z):={z}/{f(z)}$, $f\in \es ^*(A,B)$. Then $g$ has the form (\ref{eq-1-sec.2}) and satisfies the relation
$$\frac{zg'(z)}{g(z)}=1-\frac{zf'(z)}{f(z)}\prec 1-\frac{1+Az}{1+B z} =\frac{(B-A)z}{1+B z}, \quad z\in \D.
$$
Then by the definition of subordination, there exists an analytic function $ w:\D \rightarrow \overline{\D}$
with $w(0)=1$ such that
$$\frac{zg'(z)}{g(z)}=\frac{(B-A)zw(z)}{1 +Bzw(z)}, \quad z\in \D.
$$
Writing this in series form, we get
$$\sum_{k=1}^{\infty}kb_kz^{k-1}=\left((B-A)+\sum_{k=1}^{\infty}(B-A-kB)b_kz^k\right)w(z);
$$
or equivalently
$$\sum_{k=1}^{n}kb_kz^{k-1}+\sum_{k=n+1}^{\infty}c_kz^{k-1} =\left((B-A)+\sum_{k=1}^{n-1}(B-A-kB)b_kz^k\right)w(z)
$$
for certain coefficients $c_k$. Since $|w(z)|<1$ in $\D$, Parseval-Gutzmer formula
(see also Clunie's method \cite{Clu59} and \cite{CK60,Rob70,Rog43}), we obtain
$$\sum_{k=1}^{n}k^2|b_k|^2{r}^{2k-2}
\leq |A-B|^2 +\sum_{k=1}^{n-1}|B-A-kB|^2|b_k|^2{r}^{2k}.
$$
or equivalently,
\begin{equation}\label{lem-eq1}
\sum_{k=1}^{n} k^2 |b_k|^2{r}^{2k-2}
 -\sum_{k=1}^{n-1} |B-A-kB|^2  |b_k|^2{r}^{2k}
 \leq |A-B|^2.
\end{equation}
If we take $r=1$ and allow $n\rightarrow \infty$, then we obtain
the desired inequality
$$\sum_{k=1}^{\infty} \left(k^2-|B-A-kB|^2 \right)|b_k|^2 \leq |A-B|^2 .
$$
This completes the proof of our lemma.
\end{proof}
\bigskip

\begin{lemma}\label{lem2}
Let $0<|A|\leq 1$ and  $f\in \es ^*(A,0)$. For $|z|<r,$ suppose that
$$\frac{z}{f(z)}= 1+ \sum_{k=1}^{\infty}b_kz^k ~\mbox{and}~ e^{-Az}
=1+ \sum_{k=1}^{\infty}c_kz^k,\quad  r\in (0,1].
$$
Then
\begin{equation}\label{eq5}
\sum_{k=1}^N k|b_k|^2 {r}^{2k} \leq \sum_{k=1}^{N}k|c_k|^2 {r}^{2k} 
\end{equation}
holds for each $N\in \mathbb{N}$.
\end{lemma}
\begin{proof}
Clearly, it suffices to prove the lemma for $0<A\leq 1$.
From Lemma~\ref{lem1}, using the equation (\ref{lem-eq1}) for $B=0$, and then multiplying the
resulting equation by $r^2$ on both sides shows that	
\begin{equation}\label{eq6}
\sum_{k=1}^{n-1}(k^2-A^2r ^2)|b_k|^2{r}^{2k}+n^2|b_n|^2{r}^{2n}
\leq  A^2 r^2.
\end{equation}
The function $e^{Az}$ clearly shows that the equality, when $n\to \infty$, in $(\ref{eq6})$
attains with $b_k=c_k.$

\noindent
{ \bf Step-I: \, Cramer's Rule.}\\
We consider the inequalities corresponding to (\ref{eq6}) for $n = 1,\ldots, N$
and multiply the $n$th coefficient
by a factor $\lambda_{n,N}$. These factors are chosen in such a way
that the addition of the left sides
of the modified inequalities results the left side of (\ref{eq5}).
For the calculation of the
factors $\lambda_{n,N}$ we get the following system of linear equations
\begin{equation}\label{eq7}
k=k^2\lambda_{k,N}+\sum_{n=k+1}^{N} {\lambda}_{n,N} (k^2-A^2r ^2),
\, k = 1,\ldots, N.
\end{equation}

Since the matrix of this system is an upper triangular matrix with positive integers
as diagonal elements, the solution of this system is uniquely determined.
Cramer's rule allows us to write the solution of the system (\ref{eq7}) in the form
$$\lambda_{n,N}=\frac{((n-1)!)^2}{(N! )^2} \,
\mbox{Det$\,A_{n,N}$},
$$
where $A_{n,N}$ is the $(N-n+1)\times(N-n+1)$ matrix constructed as follows:\\
$$A_{n,N}= \left[ {\begin{array}{cccc}
   n & n^2-A^2r^2 & \cdots & n^2-A^2 r^2\\
   n+1 & (n+1)^2 & \cdots & (n+1)^2 -A^2 r^2\\
   \vdots & \vdots & \vdots & \vdots \\
   N & 0 & \cdots & N^2\\
  \end{array} } \right].
$$
Determinants of these matrices can be obtained by expanding according to Laplace's rule
with respect to the last row, wherein the first coefficient is $N$ and the last one is $N^2$.
The rest of the entries are zeros. This expansion and a mathematical induction results in
the following formula. If $k\leq N-1$, then
$$\lambda_{k,N}=\lambda_{k,N-1}-\frac{1}{N}\left(1-\frac{A^2r^2}{k^2}\right)
    \prod_{m=k+1}^{N-1}\left(\frac{A^2r^2}{m^2}\right).
$$
For fixed $k\in \mathbb{N}$ and $N\geq k,$
we see that the sequence $\{\lambda_{k,N}\}$ is strictly non-increasing, i.e. $\lambda_{k,N}-\lambda_{k,N-1}<0$ with
\begin{equation}\label{eq9}
 \lambda_k:=\lim_{N\rightarrow \infty}\lambda_{k,N}=\frac{1}{k}-
 \left(1-\frac{A^2r^2}{k^2}\right)
 \sum_{n=k+1}^{\infty}\frac{1}{n}\prod_{m=k+1}^{n-1} \left(\frac{A^2r^2}{m^2}\right).
\end{equation}

To prove that $\lambda_{k,N} > 0$ for all $N \in \mathbb{N}, 1 \leq k \leq N$,
it is adequate to show that $\lambda_{k}\geq 0$
for $k \in \mathbb{N}$. This will be completed in Step II. But before that we want to remark that
the proof of the said inequality is sufficient for the proof of the theorem, since,
as we remarked for (\ref{eq6}), equality holds for $b_k = c_k $.

\noindent
{\bf Step-II: \,Positivity of the Multipliers.}\\
Let for an abbreviation
$$S_k=\sum_{n=k+1}^{\infty}\frac{1}{n}\prod_{m=k+1}^{n-1} \left(\frac{A^2r^2}{m^2}\right), \quad k\in \mathbb{N}.
$$
We now show that
$$S_k\leq \frac{1}{k\left(1-\frac{A^2r^2}{k^2}\right)}.
$$
From the relation (\ref{eq9}), we have
$$\lambda_k=\frac{1}{k}-S_k+\left(\frac{A^2r^2}{k^2}\right)S_k.
$$
Again set for an abbreviation
$$ T_k= \frac{1}{k}+\left(\frac{A^2r^2}{k^2}\right)S_k.
$$
It is enough to show that
\begin{equation}\label{eq10}
 T_k\leq  \frac{1}{k\left(1-\frac{A^2r^2}{k^2}\right)}.
\end{equation}
 To show (\ref{eq10}) we use the inequality
 \begin{equation}\label{eq11}
 \frac{1}{n {\left(1-\frac{A^2r^2}{n^2}\right)}}>\frac{1}{(n+1){\left(1-\frac{A^2r^2}{(n+1)^2}\right)}}
\end{equation}
and the identity
\begin{equation}\label{eq12}
 \frac{1}{n\left(1-\frac{A^2r^2}{n^2}\right)}=\frac{1}{n}+\frac{\frac{A^2r^2}{n^2}}{n\left(1-\frac{A^2r^2}{n^2}\right)}
\end{equation}
which are admissible for each $n\in \mathbb{N}$.
Repeated application of (\ref{eq11}) and (\ref{eq12})
for $n=k,k+1,\ldots, P$ results the inequality
$$\frac{1}{k\left(1-\frac{A^2r^2}{k^2}\right)}> \sum_{n=k}^{P}\frac{1}{n}\prod_{m=k}^{n-1}\left(\frac{A^2r^2}{m^2}\right)
   +\frac{\prod_{m=k}^{P}\left(\frac{A^2r^2}{m^2}\right)}{P\left(1-\frac{A^2r^2}{P^2}\right)}
   =: S_{k,P}+R_{k,P},\quad ~\mbox{for $k\leq P$}.
$$
Since $R_{k,P}>0$ and taking limit $P\rightarrow \infty$, we obtain
$$\frac{1}{k\left(1-\frac{A^2r^2}{k^2}\right)}  \geq \lim_{P\rightarrow \infty}S_{k,P}=
\sum_{n=k}^{\infty}\frac{1}{n}\prod_{m=k}^{n-1}\left(\frac{A^2r^2}{m^2}\right).
$$
Hence, we get the relation~(\ref{eq10}).
The proof of our lemma is complete.
\end{proof}



\bigskip

\begin{lemma}\label{lem3}
Let $-1\leq B<0$, $A\neq B$ and  $f \in \es ^*(A,B)$. For $|z|<r,$ suppose that
$$\frac{z}{f(z)}= 1+ \sum_{k=1}^{\infty}b_kz^k ~\mbox{and}~ (1-Bz)^{1-A/B}
=1+ \sum_{k=1}^{\infty}c_kz^k,\quad  r\in (0,1].
$$
Then
\begin{equation}\label{eq5-lem3}
\sum_{k=1}^N k|b_k|^2 {r}^{2k} \leq \sum_{k=1}^{N}k|c_k|^2 {r}^{2k},\, N\in \mathbb{N}
\end{equation}
is recognized.
\end{lemma}
\begin{proof}
As in the proof of Lemma \ref{lem2},
we may rewrite (\ref{lem-eq1}) in the form
\begin{equation}\label{eq6-lem3}
\sum_{k=1}^{n-1} \left(k^2-|k-\phi|^2B^2r^2 \right)|b_k|^2{r}^{2k}+n^2|b_n|^2{r}^{2n}
\leq B^2|\phi|^2r^2,
\end{equation}
where  $\phi:=1-A/B$. The function $(1-Bz)^{1-A/B}$ clearly shows that the equality, when $n\to \infty$, in $(\ref{eq6-lem3})$
attains with $b_k=c_k.$

Rest of the proof is divided into two steps.

\noindent
{\bf Step-I: \,Cramer's Rule.}\\
We consider the inequalities corresponding to $(\ref{eq6-lem3})$ for $n = 1,\ldots, N$ and
multiply the $n$th coefficient
by a factor $\lambda_{n,N}$. These factors are chosen in such a way that the addition
of the left sides
of the modified inequalities results the left side of (\ref{eq5-lem3}).
For the calculation of the
factors $\lambda_{n,N}$ we get the following system of linear equations
\begin{equation}\label{eq7-lem3}
k=k^2\lambda_{k,N}+\sum_{n=k+1}^{N} {\lambda}_{n,N} \left(k^2-|k-\phi|^2B^2{r}^2 \right),
\, k = 1, \ldots, N.
\end{equation}
Since the matrix of this system is an upper triangular matrix with positive integers
as diagonal elements, the solution of this system is uniquely determined.
Cramer's rule allows us to write the solution of the system (\ref{eq7-lem3}) in the form
$$\lambda_{n,N}=\frac{((n-1)!)^2}{(N! )^2} \,
\mbox{Det$\,A_{n,N}$},
$$
where $A_{n,N}$ is the $(N-n+1)\times(N-n+1)$ matrix constructed as follows:\\
$$A_{n,N}=  \left[ {\begin{array}{cccc}
   n & n^2-|n-\phi|^2 B^2 r^2 & \cdots & n^2-|n-\phi|^2 B^2 r^2\\
   n+1 & (n+1)^2 & \cdots & (n+1)^2 -|n+1-\phi|^2 B^2 r^2\\
   \vdots & \vdots & \vdots & \vdots \\
   N & 0 & \cdots & N^2\\
  \end{array} } \right].
$$
Determinants of these matrices can be found by expanding according to Laplace's rule
with respect to the last row, wherein the first coefficient is $N$ and the last one is $N^2$.
The rest of the entries are zeros. This expansion and a mathematical induction results in
the following formula. If $k\leq N-1$, then
$$\lambda_{k,N}=\lambda_{k,N-1}-\frac{1}{N}\left(1-\left|1-\frac{\phi}{k}\right|^2B^2r^2\right)
    \prod_{m=k+1}^{N-1}\left|1-\frac{\phi}{m}\right|^2B^2 r^2.
$$
Set as an abbreviation $\displaystyle U_k= 1-\left|1-\phi/k\right|^2 B^2r^2$, we get
\begin{equation}\label{eq8-lem3}
\lambda_{k,N}=\lambda_{k,N-1}-\frac{1}{N} U_k\prod_{m=k+1}^{N-1} (1-U_m).
\end{equation}
Note that $U_k$ in (\ref{eq8-lem3}) may be positive as well as negative for all $k\in \mathbb{N}$.
We investigate it by including here a table (see Table~3).
\begin{center}
\begin{tabular}{|c|l|l|l|l|l|}
\hline
 \textbf{$k$} & \textbf{$A$} & \textbf{$B$} &  \textbf{ $r$} & \textbf{$U_k$}  \\
\hline
$1$ & $2+i$ & all & $0.5$ & $-5.25$ \\
\hline
$1$ & $1+i$ & all & $0.4$ & $0.2$ \\
\hline
$2$ & $-2+i$ & $-1$ & $0.5$ & $0.375$  \\
\hline
$2$ & $-2+i$ & $-1$ & $0.8$ & $-0.6$ \\
\hline
\end{tabular}
\end{center}
\hspace{7.5cm} Table~3

{\bf Case (i):} Suppose that $U_k$ is negative.

From the relation (\ref{eq8-lem3}), we see that for fixed $k\in \mathbb{N}, k \leq N-1$,
the sequence $\{\lambda_{k,N}\}$ is strictly non-decreasing, i.e.
$$\lambda_{k,N}-\lambda_{k,N-1}>0
$$
so that
$$ \lambda_{k,N}>\lambda_{k,N-1}>\cdots >\lambda_{k,k}=1/k>0,
$$
and thus $ \lambda_k\ge 0$ when $N\rightarrow \infty$ as required.

{\bf Case (ii):} \, Suppose that $U_k$ is positive.

For fixed $k\in \mathbb{N},~ N\geq k,$ the sequence $\{\lambda_{k,N}\}$ is strictly non-increasing,
i.e. $\lambda_{k,N}-\lambda_{k,N-1}<0$ with
\begin{equation}\label{eq9-lem3}
 \lambda_k:=\lim_{N\rightarrow \infty}\lambda_{k,N}=\frac{1}{k}-U_k\sum_{n=k+1}^{\infty}\frac{1}{n}\prod_{m=k+1}^{n-1} (1-U_m).
\end{equation}
For all $N \in \mathbb{N}, 1 \leq k \leq N$, to prove that $\lambda_{k,N} > 0$,
it is sufficient to prove $\lambda_{k}\geq 0$
for $k \in \mathbb{N}$. This will be completed in Step II.
But before that we want to annotate that
the proof of the said inequality is sufficient for the proof of the
theorem, since, as we noted in the beginning of the proof, equality is received for $b_k =c_k $.

\noindent
{\bf Step-II: \,Positivity of the Multipliers.}\\
Let for an abbreviation
$$S_k=\sum_{n=k+1}^{\infty}\frac{1}{n}\prod_{m=k+1}^{n-1} (1-U_m), \quad k\in \mathbb{N}.
$$
We now prove that
$$S_k\leq \frac{1}{kU_k}.
$$
From the relation (\ref{eq9-lem3}), we get
$$ \lambda_k=\frac{1}{k}-S_k+(1-U_k)S_k.
$$
Again set for an abbreviation
$$ T_k=\frac{1}{k}+(1-U_k)S_k.
$$
It is enough to prove that
\begin{equation}\label{eq10-lem3}
 T_k\leq \frac{1}{kU_k}.
\end{equation}
 To prove (\ref{eq10-lem3}) we use the inequality
 \begin{equation}\label{eq11-lem3}
 \frac{1}{nU_n}>\frac{1}{(n+1)U_{n+1}}
\end{equation}
and the identity
\begin{equation}\label{eq12-lem3}
 \frac{1}{nU_n}=\frac{1}{n}+\frac{1-U_n}{nU_n}
\end{equation}
which are valid for each $n\in \mathbb{N}$.
Repeated application of (\ref{eq11-lem3}) and (\ref{eq12-lem3})
for $n=k,k+1,\ldots, P$ results the inequality
$$\frac{1}{kU_k}> \sum_{n=k}^{P}\frac{1}{n}\prod_{m=k}^{n-1}(1-U_m)
   +\frac{\prod_{m=k}^{P}(1-U_m)}{PU_P}=: S_{k,P}+R_{k,P},\quad ~\mbox{for $k\leq P$}.
$$
Since $R_{k,P}>0$, taking the limit as $P\rightarrow \infty$ we obtain
$$\frac{1}{kU_k}\geq \lim_{P\rightarrow \infty}S_{k,P}= \sum_{n=k}^{\infty}\frac{1}{n}
\prod_{m=k}^{n-1}(1-U_m),
$$
and we complete the inequality~(\ref{eq10-lem3}).
This completes the proof of Lemma~\ref{lem3}.
\end{proof}

\section{Proofs of the main results} \label{sec3}
\subsection*{Proof of Theorem~\ref{thm1}}
Let $f \in \es ^*(A,0)$. By the definition of the class $\es ^*(A,0)$, it suffices to assume that
$0<A\leq 1$ and
$$ \frac{zf'(z)}{f(z)} \prec 1-Az =\frac{zk'_{A,0}(z)}{k_{A,0}(z)}, ~z\in \D.
$$
By the subordination principle, we obtain that $z/f(z) \prec e^{-Az}$ which in terms of the Taylor coefficients may be
written as
$$ 1+ \sum_{n=1}^{\infty}b_nz^n \prec e^{-Az}= 1+ \sum_{n=1}^{\infty}c_nz^n ,\quad  c_n= (-1)^n \frac{A^n}{n! }.
$$
 %
By Lemma~\ref{lem2}, we have
$$ \sum_{n=1}^N n|b_n|^2 {r}^{2n} \leq \sum_{n=1}^{N}n|c_n|^2 {r}^{2n},\, N\in \mathbb{N},\, r\in (0,1),
$$
which implies that
$$
\Delta \left(r,\frac{z}{f}\right) = \pi \sum_{n=1}^{\infty}n|b_n|^2 {r}^{2n} \leq \Delta \left(r,\frac{z}{k_{A,0}}\right)
=\pi \sum_{n=1}^{\infty}n|c_n|^2 {r}^{2n}.
$$
We claim that
$$\pi \sum_{n=1}^{\infty}n|c_n|^2 {r}^{2n} =E_{A,0}(r),
$$
where $E_{A,0}(r) =  \pi {A}^2{r}^2 {}_0F_1(2;{A}^2{r}^2)$  with $0<A\leq 1$. To prove the claim, we observe that
\begin{align*}
\pi^{-1}\Delta\left(r,\frac{z}{k_{A,0}}\right) & =
\sum_{n=1}^{\infty}n\frac{A^{2n}}{(n!)^2 } {r}^{2n}  \\
& = {A}^2{r}^2\sum_{n=0}^{\infty}
\frac{1}{(2)_n(1)_n}{A}^{2n}{r}^{2n} \\
& =  A^2{r}^2 {}_0F_1(2;{A}^2{r}^2)\\
 & := {\pi}^{-1}E_{A,0}(r)
\end{align*}
and thus,
$$\Delta \left(r,\frac{z}{f}\right) \leq \Delta \left(r,\frac{z}{k_{A,0}}\right) =E_{A,0}(r).
$$
The equality case is obvious from $z/k_{A,0}(z) =e^{-Az}$. The proof of the theorem is complete.
\hfill$\Box$

\subsection*{Proof of Theorem~\ref{thm2}}
Suppose $f \in \es ^*(A,B)$,  $-1\leq B<0$ and $A \neq B$. Then by setting $g(z):={z}/{f(z)}$, we write
$$\frac{zg'(z)}{g(z)}=1-\frac{zf'(z)}{f(z)}\prec 1-\frac{1+Az}{1+B z} =\frac{(B-A)z}{1+B z}, \quad z\in \D.
$$
By a well-known subordination result, we get
$$ g(z)=\frac{z}{f(z)} \prec (1+Bz)^{1-\frac{A}{B}}= \frac{z}{k_{A,B}(z)},
$$
where $k_{A,B}$ is defined by \eqref{eq2-ext}.  If
$$ \frac{z}{f(z)}= 1+ \sum_{n=1}^{\infty}b_nz^n ~\mbox{and}~\ \frac{z}{k_{A,B}(z)}
 =1+ \sum_{n=1}^{\infty}c_nz^n,\quad |z|<r,
$$
then Lemma~\ref{lem3} gives that
$$ \sum_{n=1}^N n|b_n|^2 {r}^{2n} \leq \sum_{n=1}^{N}n|c_n|^2 {r}^{2n}
$$
for each $ N\in \mathbb{N}$ and $r\in (0,1]$. Allowing $N\rightarrow \infty$, we obtain
$$ \Delta \left(r,\frac{z}{f}\right) = \sum_{n=1}^{\infty}n|b_n|^2 {r}^{2n} \leq \Delta \left(r,\frac{z}{k_{A,B}}\right)
= \sum_{n=1}^{\infty}n|c_n|^2 {r}^{2n}.
$$
Clearly,
$$\displaystyle c_n= (-1)^n\frac{(\zeta)_n}{(1)_n} {B}^n ~\mbox{  with }~\zeta =(A/B)-1.
$$
Now, applying the area formula \eqref{ex-eq1} for the function $\displaystyle z/k_{A,B}(z)$, we obtain that
\begin{align*}
\pi^{-1}\Delta\left(r,\frac{z}{k_{A,B}}\right) & =
\sum_{n=1}^{\infty}n|c_n|^2 {r}^{2n}, \quad |z|<r \\
 & = \sum_{n=1}^{\infty}n \frac{(\zeta)_n(\overline{\zeta})_n}{(1)_n(1)_n}{B}^{2n}{r}^{2n} \\
& = |\zeta|^2{B}^2{r}^2\sum_{n=0}^{\infty}
\frac{(\zeta+1 )_n(\overline{\zeta}+1)_n}{(2)_n(1)_n}{B}^{2n}{r}^{2n} \\
& =  |\overline{A}-B|^2{r}^2 {}_2F_1(A/B,\overline{A}/B;2;{B}^2{r}^2)\\
& := {\pi}^{-1}E_{A,B}(r),
\end{align*}
and the proof of Theorem~\ref{thm2} is complete.
\hfill$\Box$

\section{Concluding Remarks}
It would be interesting to solve the analog of Yamashita's extremal problem \eqref{eq3-ext} for
many interesting geometric subclasses of functions from $\es$. For example, determine the
analog of Theorems~\ref{thm1} and \ref{thm2} when $zf'$ belongs to the class  $\es ^*(A,B)$
and also for functions $f$ in the Bazilevi${\rm \check{c}}$ class or for functions convex 
in some direction.

\bigskip
\noindent
{\bf Acknowledgements.} This research was started in the mid 2013 when the
second and third author were visiting ISI Chennai Centre.
The third author acknowledges the support of National Board for Higher Mathematics,
Department of Atomic Energy, India (grant no. 2/39(20)/2010-R$\&$D-II).


\begin{thebibliography}{99}

\bibitem{Clu59} {J. G. Cluine},
On meromorphic schlicht functions,
{\em J. London Math. Soc.},
\textbf{34}(1959), 215--216.

\bibitem{CK60} { J. G. Clunie and F. R. Keogh},
On starlike and convex schlicht functions,
{\em J. London Math. Soc.},
\textbf{35}(1960), 229--233.

\bibitem{Dur83} { P. L. Duren},
{\em Univalent Functions}, Springer-Verlag, 1983.


\bibitem{Good83} { A. W. Goodman},
{\em Univalent Functions}, Vol. 1-2, Mariner, Tampa, Florida, 1983.

\bibitem{Jan73} { W. Janowski},
Some extremal problems for certain families of analytic functions,
{\em Ann. Polon. Math.},
\textbf{28}(1981), 297--326.

\bibitem{Lib67} { R. J. Libera},
Univalent $\alpha$-spiral functions,
{\em Canad. J. Math.},
\textbf{19}(1967), 449--456.



\bibitem{Nev21} { R. Nevalinna},
$\ddot{U}$ber die konforme Abbildung von Sterngebieten,
{\em $\ddot{O}$resikt av Finska Vetenskaps Soc. $\ddot{F}$orh.},
\textbf{63A}, no. 6(1921), 1--21.


\bibitem{OPW13} { M. Obradovi\'{c}, S. Ponnusamy, and K.-J. Wirths},
A proof of Yamashita's conjecture on area integral,
{\em Comput. Methods Funct. Theory}, \textbf{13}(2013), 479--492.

\bibitem{OPW14} { M. Obradovi\'{c}, S. Ponnusamy, and K.-J. Wirths},
Integral means and Dirichlet integral for analytic functions,
Mathematische Nachrichtten (2014), 9 pages, To appear.

\bibitem{Pad68} { K. S. Padmanabhan},
On certain classes of starlike functions in the unit disk,
{\em J. Indian Math. Soc.}, (N.S.)
\textbf{32}(1968), 89--103.


\bibitem{PW13} { S. Ponnusamy, and K.-J. Wirths},
On the problem of Gromova and Vasil'ev on integral means, and Yamashita's
conjecture for spirallike functions,
{\em Ann. Acad. Sci. Fenn. Ser. AI Math.}  \textbf{39}(2014), 721--731.

\bibitem{Rai60} { E. D. Rainville},
{\em Special Functions}, The Macmillan Company, New York, 1960.

\bibitem{Rob36} { M. S. Robertson},
On the theory of univalent functions,
{\em Annals of Mathematics},
\textbf{37}(1936), 374--408.

\bibitem{Rob70} { M. S. Robertson},
Quasi-subordination and coefficient conjectures,
{\em J. Bull. Amer. Math. Soc.},
\textbf{76}(1970), 1--9.

\bibitem{Rog43} { W. Rogosinski},
On the coefficients of subordinate functions,
{\em Proc. London Math. Soc.},
\textbf{48}(2)(1943), 48--82.

\bibitem{Sil78} { H. Silverman},
Subclass of starlike functions,
{\em Rev. Roum. Math. Pure Appl.},
\textbf{33}(1978), 1093--1099.

\bibitem{Sin68} { R. Singh},
On a class of starlike functions,
{\em J. Indian Math. Soc.},
\textbf{32}(1968), 208--213.

\bibitem{SiSi74} { R. Singh and V. Singh},
On a class of bounded starlike functions,
{\em Indian J. Pure Appl. Math.},
\textbf{5}(1974), 733--754.

\bibitem{SS14} { S. K. Sahoo and N. L. Sharma},
On area integral problem for analytic functions in the starlike family, preprint
(see the preprint available at {\em arXiv:1405.0469 [math.CV]}).

\bibitem{Spacek-33} { L. ${\rm \check{S}pa\check{c}ek}$},
Contribution $\rm \grave{a}$ la th$\rm \acute{e}$orie des fonctions univalentes
(in Czech),
{\em $\rm \check{C}$asop P$\rm \check{e}$st. Mat.-Fys.,} {\bf 62}(1933), 12--19.


\bibitem{Yam90} { S. Yamashita},
Area and length maxima for univalent functions,
{\em Proc. London Math. Soc.},
\textbf{41}(2)(1990), 435--439.
\end{thebibliography}
\end{document}